\documentclass[10pt]{article}

\usepackage[a4paper,margin=1in]{geometry}
\usepackage{amsmath,amssymb,amsthm}
\usepackage[T1]{fontenc}
\usepackage{url}
\usepackage{enumitem}

\numberwithin{equation}{section}

\newtheorem{theorem}{Theorem}[section]
\newtheorem{lemma}[theorem]{Lemma}

\newcommand{\C}{\mathbb C}

\newcommand{\N}{\mathbb N}

\newcommand{\norm}[1]{\left\|#1\right\|}

\title{A Bounded Closed Convex Set in $c_0(\mathbb{C})$ without\\[2pt]
Modulus Support Functionals}
\author{Xin Luo}
\date{}

\begin{document}

\maketitle

\begin{abstract}
We construct a bounded closed convex set in $c_0(\mathbb{C})$ on which no
nonzero continuous complex-linear functional attains its maximum modulus. Consequently, 
this answers in the negative the two longstanding questions raised by
Deville, Godefroy and Zizler in
\cite[p. 35]{DGZ}. The construction uses a small
perturbation of the identity on $c_0(\mathbb{R})$.
\end{abstract}

\medskip
\noindent\textbf{Keywords.} Complex Banach space, modulus support
functional, Bishop--Phelps theorem, convex set.


\section{Introduction}
\indent

For a nonempty bounded closed convex set $K$ in a Banach space $X$, denote
the set of modulus support functionals for $K$ by
\begin{equation}
\label{eq:modulus-support}
S(K):=\left\{f\in X^*\setminus\{0\}: |f(x_0)|=\sup_{x\in K}|f(x)|\text{ for some }x_0\in K\right\}.
\end{equation}
The classical Bishop--Phelps theorem states that $S(K)$ is norm dense in
$X^*$ when $X$ is a real Banach space \cite{BishopPhelps1961,BishopPhelps}. 
For complex Banach spaces, the situation is more subtle. 
Bourgain \cite{Bourgain1977} proved that modulus support functionals are norm
dense when the space has the Radon--Nikod\'{y}m property. In general, however,
Lomonosov \cite{Lomonosov} constructed a bounded closed convex set in a
complex Banach space with no modulus support functionals, but his construction
does not cover the classical sequence space $c_0(\mathbb{C})$,
the space of all complex sequences converging to zero.
Moreover, $c_0(\mathbb{C})$ does not have the
Radon--Nikod\'{y}m property \cite[Chapter VII]{DiestelUhl}, so Bourgain's
theorem does not apply to it. Deville, Godefroy and Zizler
\cite[p. 35]{DGZ} asked whether every nonempty
bounded closed convex set in $c_0(\mathbb{C})$ admits a nonzero modulus
support functional and whether such functionals are norm dense in
$c_0(\mathbb{C})^*$ (see also \cite[Problem 208]{GMZ})?

In $c_0(\mathbb{C})$, Kadets, L\'{o}pez, Mart\'{\i}n and Werner \cite{KLMW}
studied the closed convex hull of the vectors $(z,z^2,z^3,\ldots)$,
$|z|<1$, and Golinskii and Kadets \cite[Theorem 4.1 and Remarks
4.2--4.3]{GK} showed that this set does admit nonzero modulus support
functionals. This provides a positive result for a particular convex set in
$c_0(\mathbb{C})$, but does not resolve the questions above.

In this paper, we construct a bounded closed convex set in
$c_0(\mathbb{C})$ which has no modulus support functionals. Let
\[
 X=c_0(\mathbb{C})
	 :=\left\{x=(x_n)_{n\in\mathbb{N}}\in\mathbb{C}^{\mathbb{N}}:
	 x_n\longrightarrow 0\right\},
 \qquad
 \norm{x}_\infty=\sup_{n\geq 1}|x_n|,
\]
and denote its closed unit ball by $B_X$. By the standard duality
$c_0(\C)^*=\ell_1(\C)$ (see, for example, \cite[p. 76]{Conway}),
we identify $X^*$ with $\ell_1(\C)$ via the complex-linear pairing
\begin{equation}
\label{eq:dual-pairing}
 f_a(x)=\sum_{n=1}^{\infty}a_nx_n,
 \qquad
 a\in\ell_1(\C),\quad x\in X.
\end{equation}
For a nonempty bounded subset $K\subset X$, put
\[
M(K)=\left\{a\in\ell_1(\C):
 |f_a(x_0)|=\sup_{x\in K}|f_a(x)|
 \text{ for some }x_0\in K\right\}.
\]

\begin{theorem}
There exists a bounded closed convex set $\mathcal{C}\subset X$ such that
\begin{equation}
\label{eq:main-conclusion}
 \mathcal{C}=-\mathcal{C},
 \qquad
 \frac12 B_X\subset\mathcal{C}\subset \frac{\sqrt{13}}2B_X,
 \qquad
 M(\mathcal{C})=\{0\}.
\end{equation}
\end{theorem}

Since $M(\mathcal{C})=\{0\}$, both the existence and the density questions
in \cite[Problem 208]{GMZ} have negative answers.

The proof is given in Section 2. We construct a real-linear isomorphism
$T$ of $c_0(\mathbb{R})$ and take $\mathcal{C}=T(B_E)+iB_E$, where $B_E$
is the closed unit ball of $c_0(\mathbb{R})$. If $f_a$
attains its maximum modulus on $\mathcal{C}$, variations
in the two factors force both $(a_n)$ and $(f_a(Te_k))$ to have finite
support. The choice of $T$ makes these two conditions possible only when
$a=0$.

\section{Proof of Theorem 1.1}

Let $E=c_0(\mathbb{R})$, and denote its closed unit ball by $B_E$. We regard
$E$ as the real subspace of $X$ and write $e_n$ for the standard coordinate
vectors. We first construct the operator needed in the proof.

\begin{lemma}
There exists a bounded real-linear isomorphism $T:E\to E$ such that
\begin{equation}
\label{eq:operator-bounds}
 \norm{T}\leq\frac32,
 \qquad
 \norm{T^{-1}}\leq2,
\end{equation}
and the following property holds: if $a\in\ell_1(\C)$ has finite
support and $f_a(Te_k)=0$ for all sufficiently large $k$, then $a=0$.
\end{lemma}

\begin{proof}
For $n\geq1$, set
\[
 t_n=(n+2)^{-1}
\]
and define
\begin{equation}
\label{eq:operator-A}
 (Au)_n:=\sum_{k=1}^{\infty}t_n^k u_k,
 \qquad u\in E.
\end{equation}
The series is obviously convergent, and
\begin{equation}
\label{eq:A-estimate}
 |(Au)_n|
 \leq
 \frac{t_n}{1-t_n}\norm{u}_\infty
 =\frac{1}{n+1}\norm{u}_\infty.
\end{equation}
It follows that $Au\in E$ and $\norm{A}\leq1/2$. Let $T:=I+A$, then $T$ is invertible, 
and the Neumann series gives
\[
 T^{-1}=\sum_{j=0}^{\infty}(-A)^j,
 \qquad
 \norm{T^{-1}}
 \leq\frac1{1-\norm{A}}
 \leq2.
\]
Since $\norm{T}\leq1+\norm{A}\leq3/2$, the norm estimates follow.

Suppose that $a\neq0$ has finite support and that
$f_a(Te_k)=0$ for all sufficiently large $k$. Put
\[
 F=\{n\in\mathbb{N}_+:a_n\neq0\},
 \qquad n_0=\min F.
\]
By \eqref{eq:operator-A},
\[
 (Te_k)_n=\delta_{nk}+t_n^k.
\]
Thus, for all sufficiently large $k>\max F$,
\[
 0=f_a(Te_k)=\sum_{n\in F}a_nt_n^k.
\]
Dividing by $t_{n_0}^k$ gives
\[
 0=a_{n_0}
 +\sum_{\substack{n\in F\\n>n_0}}
 a_n\left(\frac{t_n}{t_{n_0}}\right)^k.
\]
The sum is finite, and $0<t_n/t_{n_0}<1$ whenever $n>n_0$. Letting
$k\to\infty$ yields $a_{n_0}=0$, a contradiction. This completes
the proof.
\end{proof}

\begin{proof}[Proof of Theorem 1.1]
Let $T$ be the operator in Lemma 2.1, and define
\begin{equation}
\label{eq:convex-body}
 \mathcal{C}=\{Tu+iv:u,v\in B_E\}.
\end{equation}
Since the representation $x=Tu+iv$ is unique, we also have
\begin{equation}
\label{eq:convex-body-characterization}
 \mathcal{C}=\left\{x\in X:
 \norm{T^{-1}\operatorname{Re}x}_\infty\leq1,
 \quad
 \norm{\operatorname{Im}x}_\infty\leq1
 \right\}.
\end{equation}
The maps $x\mapsto T^{-1}\operatorname{Re}x$ and
$x\mapsto\operatorname{Im}x$ are continuous and real-linear. Therefore
$\mathcal{C}$ is closed and convex. The symmetry $\mathcal{C}=-\mathcal{C}$ is immediate.

For $x=Tu+iv\in\mathcal{C}$, \eqref{eq:operator-bounds} gives
\[
 |x_n|^2=|(Tu)_n|^2+|v_n|^2
 \leq\left(\frac32\right)^2+1
 =\frac{13}{4}.
\]
Consequently,
\[
 \mathcal{C}\subset\frac{\sqrt{13}}2B_X.
\]
Conversely, if $\norm{x}_\infty\leq1/2$, then
\[
 \norm{T^{-1}\operatorname{Re}x}_\infty
 \leq2\norm{x}_\infty\leq1,
 \qquad
 \norm{\operatorname{Im}x}_\infty\leq\frac12.
\]
By \eqref{eq:convex-body-characterization}, $x\in\mathcal{C}$. This proves the second inclusions in \eqref{eq:main-conclusion}.

It remains to prove that $M(\mathcal{C})=\{0\}$. Let $a\in M(\mathcal{C})$, and choose
$u_0,v_0\in B_E$ such that, for $x_0=Tu_0+iv_0$,
\[
 |f_a(x_0)|=\sup_{x\in\mathcal{C}}|f_a(x)|.
\]
Put $z=f_a(x_0)$. We will use the following implication: if
$|z+w|\leq|z|$ and $|z-w|\leq|z|$, then $w=0$. Indeed, this follows from
the parallelogram identity
\begin{equation}
\label{eq:parallelogram}
 \frac{|z+w|^2+|z-w|^2}{2}=|z|^2+|w|^2,
 \qquad z,w\in\C.
\end{equation}

Since $v_0\in c_0(\mathbb{R})$, there exists $N_1\in\N$ such that
$|(v_0)_n|\leq1/2$ for $n>N_1$. For every such $n$,
$v_0\pm e_n/2\in B_E$, so
\[
 x_0\pm\frac{i}{2}e_n\in\mathcal{C}.
 \qquad
 \left|z\pm\frac{i}{2}a_n\right|\leq|z|.
\]
The two inequalities, together with \eqref{eq:parallelogram} applied to
$w=ia_n/2$, imply
\begin{equation}
\label{eq:a-finite-support}
 a_n=0\qquad n>N_1.
\end{equation}

Similarly, since $u_0\in c_0(\mathbb{R})$, there exists $N_2\in\N$ such that
$|(u_0)_k|\leq1/2$ for $k>N_2$. Hence
\[
 x_0\pm\frac12Te_k
 =T\left(u_0\pm\frac12e_k\right)+iv_0\in\mathcal{C}
 \qquad k>N_2.
\]
It follows that
\[
 \left|z\pm\frac12f_a(Te_k)\right|\leq|z|,
\]
and these inequalities with \eqref{eq:parallelogram}, applied to
$w=f_a(Te_k)/2$, imply
\begin{equation}
\label{eq:Te-finite-support}
 f_a(Te_k)=0\qquad k>N_2.
\end{equation}

By \eqref{eq:a-finite-support}, $a$ has finite support. Lemma 2.1 and \eqref{eq:Te-finite-support} now imply
$a=0$. Since the zero functional belongs to $M(\mathcal{C})$, we conclude that
$M(\mathcal{C})=\{0\}$.
\end{proof}

\textbf{Statement of AI Use.} The key construction and an initial proof of the main theorem 
were generated by OpenAI's ChatGPT in response to the author's research questions. 
The author subsequently examined the arguments, independently verified the proof, 
and revised the mathematical exposition. 
The author assumes full responsibility for the final manuscript, 
including its mathematical correctness and originality.

\bigskip
\noindent Xin Luo\\
School of Mathematical Sciences\\
Xiamen University\\
Xiamen, China\\
E-mail: \url{xinluo2058@xmu.edu.cn}

\end{document}